\documentclass[12pt,twoside]{amsart}

\usepackage[pagebackref]{hyperref}

\usepackage[english]{babel}
\usepackage[all]{xy}
\usepackage{amssymb,amsmath,bbm,url,xr,a4wide,cleveref,color}
\usepackage[mathscr]{euscript}
\usepackage{mathrsfs}
\usepackage[margin=1.2in]{geometry}
\usepackage{amsmath}
\usepackage{amsthm}
\usepackage{amssymb}
\usepackage{amscd}
\usepackage{mathrsfs}
\usepackage{graphicx}
\usepackage{rotating, graphicx}
\usepackage{multirow}
\usepackage{mathtools}
\usepackage{mathrsfs}
\usepackage{changepage} 
\usepackage{paralist}

\usepackage{ulem} 

\usepackage{kantlipsum}
\usepackage{enumitem}
\usepackage{times}
\usepackage{subfig}
\usepackage{color}
\usepackage{cleveref}
\calclayout

\usepackage{enumitem}
\makeatletter
\def\namedlabel#1#2{\begingroup
    #2%
    \def\@currentlabel{#2}%
    \phantomsection\label{#1}\endgroup
}
\makeatother

\title{Moving lemmas and the homotopy coniveau tower}

\author{Fr\'ed\'eric D\'eglise, Niels Feld, and Fangzhou Jin}

\date{\number\day-\number\month-\number\year}

\newtheorem{thm}[subsubsection]{Theorem}

\newtheorem{lm}[subsubsection]{Lemma}

\newtheorem*{conj}{Conjecture}

\theoremstyle{remark}
\newtheorem{rem}[subsubsection]{Remark}

\theoremstyle{definition}
\newtheorem{df}[subsubsection]{Definition}

\newtheorem{paragr}[subsubsection]{}

\numberwithin{equation}{subsubsection}

\newcommand{\cC}{\mathcal C}

\newcommand{\cH}{\mathcal H}

\newcommand{\cW}{\mathcal W}
\newcommand{\cU}{\mathcal U}
\newcommand{\cZ}{\mathcal Z}
\newcommand{\cN}{\mathcal N}

\newcommand{\SH}{\mathrm{SH}}

\newcommand{\DM}{\mathrm{DM}}

\DeclareMathOperator{\spz}{sp}

\usepackage{mathtools}

\newcommand{\plim} { \varprojlim }

\newcommand{\NN} {\mathbf N}
\newcommand{\ZZ} {\mathbf Z}

\newcommand{\Span} {\mathbf{Span}}

\renewcommand{\AA} {\mathbf A}

\newcommand{\GG} {\mathbf{G}_m}

\newcommand{\E}{\mathbb E}

\DeclareMathOperator{\Sh}{Sh}

\newcommand{\Sm}{\mathit{Sm}}
\newcommand{\Smsm}{\mathit{Sm}^{sm}}

\newcommand{\Smsp}{\mathit{Sm}^{sp}}

\newcommand{\Smspr}{\mathit{Sm}^{spr}}

\DeclareMathOperator{\Th}{Th} 

\newcommand{\codim}{\operatorname{codim}}

\newcommand{\Id}{\operatorname{Id}}

\newcommand{\Gm}{\mathbb{G}_m}

\usepackage{pdftexcmds}

\DeclareFontFamily{U}{cbgreek}{}
\DeclareFontShape{U}{cbgreek}{m}{n}{
	<-6>    grmn0500
	<6-7>   grmn0600
	<7-8>   grmn0700
	<8-9>   grmn0800
	<9-10>  grmn0900
	<10-12> grmn1000
	<12-17> grmn1200
	<17->   grmn1728
}{}
\DeclareFontShape{U}{cbgreek}{bx}{n}{
	<-6>    grxn0500
	<6-7>   grxn0600
	<7-8>   grxn0700
	<8-9>   grxn0800
	<9-10>  grxn0900
	<10-12> grxn1000
	<12-17> grxn1200
	<17->   grxn1728
}{}

\makeatletter
\newcommand{\normalorbold}{%
	\ifnum\pdf@strcmp{\math@version}{bold}=\z@ bx\else m\fi
}
\makeatother

\begin{document}

\begin{abstract}
In this note we study the functoriality of the coniveau filtration in motivic homotopy theory via a moving lemma over a base scheme, 
 extending previous works of Levine and Bachmann-Yakerson.  
The main result is that the motivic stable homotopy category can be modeled on a smaller site,
the \emph{smooth-smooth site}.
 The proof is based on a new approach to the purity theorem of Morel-Voevodsky using
specialization maps, which turns out to hold even in absence of the 
$\AA^1$-homotopy invariance property.
 Applications to the homotopy coniveau tower and to higher Chow-Witt groups are given.
\end{abstract}

\maketitle

\tableofcontents

\section{Introduction}

\subsection{Gysin morphisms in motivic homotopy theory}

In the late 1990s, Voevodsky initiated a unification of algebraic and topological methods. Combining algebraic geometry and homotopy theory, Morel and Voevodsky developed what is now called motivic homotopy theory, the main idea of which was to apply the techniques from classical algebraic topology to the study of schemes, replacing topological deformation parameterized by $[0,1]$ with algebraic deformation parameterized by the affine line $\AA^1$.\footnote{Beware that, as in topology, there are ``strong'' (algebraic) $\AA^1$-homotopies which corresponds to the naive idea, but also ``weak'' $\AA^1$-homotopies which mix in a non trivial way the strong ones with the Nisnevich topology. The latter explains the subtlety of motivic homotopy.}
The idea of motivic homotopy came into the scene first as a background of Voevodsky's theory of mixed motives,
 then as a critical tool in his proof of the Milnor and Bloch-Kato conjectures. It led to many new concepts among which the motivic Steenrod algebra, the algebraic cobordism and the category of (mixed) motivic complexes which is now encompassed into the more general motivic stable homotopy category.


\par 

One of the main fundamental result of motivic homotopy is the following theorem of Morel and Voevodsky \cite[Theorem 2.23, Section 3]{MorelVoevodsky}:
\begin{thm}[Morel-Voevodsky purity isomorphism]
\label{thm:MV3223}
	Let $S$ be a noetherian scheme of finite dimension. For every closed immersion $Z\to X$ of smooth schemes over $S$, there exists a natural isomorphism
	$$
	X/(X-Z)
	\simeq 
	\Th(N_ZX)
	$$
in the motivic pointed unstable homotopy category $\mathscr{H}_S$ where $\Th(N_ZX) = N_ZX / (N_ZX - Z)$ is the Thom space of the normal bundle $N_ZX$ of $Z$ in $X$.
\end{thm}
Theorem~\ref{thm:MV3223} can be seen as an algebraic analogue of the theory of tubular neighborhoods of immersions of closed smooth manifolds. Likewise, it led to many interesting consequences, such as the existence of the Gysin sequence for vector bundles, or that of Gysin morphisms.

In this paper, we consider the \emph{smooth-smooth site} $\Smsm_S$, that is, the Nisnevich site on the category consisting of smooth $S$-schemes with smooth morphisms and generalize the previous theorem:
\begin{thm}[General purity theorem]
	\label{thm_main_purity}
	(see \Cref{thm_general_purity})
	Let $S$ be an excellent noetherian scheme.
 	Let $i:Z \to X$ be a closed immersion of smooth schemes over $S$. There exists a weak equivalence 
	$$
	{\mathfrak{p}_{X,Z}} :	X/(X-Z) \simeq  \Th(N_ZX)
	$$
	in the $\infty$-category of Nisnevich sheaves over the smooth-smooth site $\Smsm_S$.
\end{thm}
We emphasize the fact that Theorem~\ref{thm_main_purity} holds generally without assuming $\AA^1$-homotopy invariance. This result is used to prove our main result:
\begin{thm}[General motivic moving lemma]
	\label{thm_main_moving}
For any excellent noetherian scheme $S$, we let $\nu:\Smsm_S \rightarrow \Sm_S$ be the canonical inclusion of Nisnevich site. The restriction functor:
$$
\nu^*:\SH(\Sm_S) \rightarrow \SH(\Smsm_S)
$$
is an equivalence of monoidal $\infty$-categories.
\end{thm}
In other words, we can reconstruct the whole functoriality of a Nisnevich sheaf by simply looking at the pullbacks along smooth morphisms, provided it satisfies the $\AA^1$-localizing property and it has the structure of an infinite $\Gm$-loop space. As a concrete example, one can consider the notion of mixed Weil theory\footnote{recall examples are given by de Rham, (geometric) \'etale and rigid cohomologies} from \cite{CD2}: it is sufficient to have such a functor on smooth affine schemes with only smooth morphisms to reconstruct the full cohomology theory, and in particular the pullback with respects to arbitrary morphisms
of smooth (and in fact even singular) $S$-schemes.

\subsection{Perverse homotopy heart and Milnor-Witt cycle modules}

\par 

The notion of \emph{Milnor-Witt cycle modules} is introduced by the second-named author in \cite{Feld1} over a perfect field and studied in \cite{Feld2, Fel21, Fel21b}. Milnor-Witt cycle modules generalize Rost's theory \cite{Rost96} and form a practical framework to develop \textit{quadratic} intersection theoretic questions \cite{Fel21c} and their applications in birational geometry \cite{Fel22a}.

\par 

After slight changes, the theory of Milnor-Witt cycle modules can be extended to regular base schemes (see \cite{BHP22}). In the case of a singular case scheme, the theory is more involved and has been dealt with in \cite{DegliseFeldJin22}.

\par Our interest in Milnor-Witt cycle module lies in the fact that they can be used to compute hearts of motivic homotopy categories as described below.

Voevodsky defines a $t$-structure on the category of motivic complexes $\DM(k)$ over a perfect field $k$, the \emph{homotopy $t$-structure}, whose heart is identified with the abelian category of homotopy invariant sheaves with transfers. In the thesis of the first-named author, this heart is further identified with Rost's category of cycle modules over $k$ (see \cite[Prop. 5.6]{Deg9}), which is a far-fetched generalization of the Gersten resolution of homotopy sheaf with transfers proved by Voevodsky. Although this $t$-structure fails to be a candidate of the conjectural \emph{motivic $t$-structure} whose heart gives an abelian category of mixed motives, Voevodsky expects a way to define the latter from the former.

The homotopy $t$-structure is also defined over $\SH(k)$ by Morel, which turns out to be an analogue of the \emph{Postnikov tower} in algebraic topology. 
These constructions are later extended to a base scheme by Ayoub in \cite{Ayoub1}: the so-called \emph{perverse homotopy $t$-structure} has the distinctive feature to be obtained by gluing in the sense of Beilinson-Bernstein-Deligne-Gabber. Ayoub makes the following conjecture on the heart of the perverse homotopy $t$-structure:
\begin{conj}
	Let $S$ be a noetherian excellent finite dimensional base scheme of characteristic $0$.
	Then the heart of the perverse homotopy t-structure on $\DM(S)$
	is equivalent to the category of Rost cycles modules over $S$.
\end{conj}









In our paper \cite{DegliseFeldJin22} we prove this conjecture, and even improves it in several directions. A crucial step is given by the work \cite{BD1}, which extends Ayoub's construction to schemes endowed with a dimension function: the perverse homotopy t-structure is non-degenerate in equal characteristic $p$ with $\ZZ[1/p]$-coefficients, and in mixed characteristics with rational coefficients.
 Secondly, relying on the ground work on Milnor-Witt modules of \cite{Fel21},
 we can state and prove the conjecture when the category of motivic complexes $\DM$ is replaced by
 the stable motivic homotopy category $\SH$.

One of the main technical difficulties in the proof is to extend 
the functoriality of the coniveau filtration on objects of $\SH(S)$
 (associated to a given dimension function on $S$) to all morphisms between smooth $S$-schemes (which are always local complete intersection), 
 which is not obvious unless the morphism is flat. 
 Our main theorem \Cref{thm_main_moving} allows us to circumvent this problem and obtain the desired enhanced functoriality, as a direct corollary.

\subsection{The homotopy coniveau tower}

Let $B$ be a Noetherian separated scheme of finite Krull dimension. For a functor $E:
\Sm_B \to \textbf{Spt}$ from smooth $B$-schemes of finite type to spectra Levine constructs in \cite{Levine06} the homotopy coniveau tower
$$
\dots
\to
E^{(p+1)}(X,-)
\to 
E^{(p)}(X,-)
\to
\dots 
\to 
E^{(0)}(X,-).
$$

The $E^{(p)}(X,-)$ are simplicial spectra, whose $n$-simplices are limits of spectra with support $E^W(X\times \Delta_n)$, where $W$ runs through closed subsets of $X\times \Delta_n$ with the property that $\codim_{X×F}(W \cap (X\times F)) \geq p$ for all faces of the algebraic $n$-simplex $\Delta_n$.

Let $E^{(p/p+1)}(X,-)$ denote the homotopy cofiber of the map $E^{(p+1)}(X,-) \to E^{(p)}(X,-)$. An immediate consequence of the homotopy coniveau tower is the spectral sequence
$$
E^{p,q}_1 
:=
\pi_{-p - q}( E^{(p/p+1)}(X,-))
\Rightarrow
\pi_{-p-q}(E^{(0)}(X,-)),
$$
which in case of the K-theory spectrum yields the Bloch-Lichtenbaum spectral sequence from motivic cohomology to K-theory (see \cite{BlochLichtenbaum95}).

The main result of \cite{Levine06} yields the functoriality of the previous spectral sequence
 with respect to $\Sm_B$, where $B$ is a Dedekind scheme.
 In order to achieve this, Levine first constructs a functorial model of the presheaf of $S^1$-spectra 
$U
\mapsto
E^{(p)}(U,-)$
on the Nisnevich site $X_{Nis}$ of $X$, which is denoted by $E^{(p)}(X_{Nis},-)$. 
The construction uses a generalization of the classical method used to prove Chow’s moving lemma for cycles modulo rational equivalence. This yields functoriality for a similar spectral sequence with $E^{(p)}(X,-)$ replaced by a fibrant model of $E^{(p)}(X_{Nis},-)$. Functoriality of the spectral sequence then follows from the localization properties of $E^{(p)}(X,-)$, which were developed in \cite{Levine2001}.

\par Our main theorem \ref{thm_main_moving} can be used to replace Levine's moving lemma,
 at the cost of working $\GG$-stably.
 Moreover, we obtain a result valid over any excellent noetherian base scheme, provided one uses a (co)dimension
 function to define the coniveau filtration. As an example,
 this might be used in order to extend Bloch's theory of higher Chow groups over arbitrary base schemes.
 Note however that the localization property, even with this new approach, still remains an open problem.

\par The main constructions of \cite{Levine06} have been later on used in \cite{BachmannYakerson20} to study a motivic conservativity conjecture over a perfect base field. Particular cases of this conjecture have also been proved in \cite{Bachmann20} and improved in \cite{Fel21}. We hope to use our current results to study these problems in general.

\subsubsection*{Notations and conventions}

\begin{itemize}
\item
All schemes are assumed noetherian and finite dimensional.
\item
We fix an excellent base scheme $S$ and ring of coefficients $R$. If not stated otherwise, all schemes and morphisms of schemes are defined over $S$.
\item
For a scheme $T$, a \emph{point} (resp. \emph{trait}, \emph{singular trait}) of $T$ will be a morphism  of schemes essentially of finite type from the spectrum of 
a field (resp. discrete valuation ring, local ring of dimension $1$) to $T$. Morphisms of them are morphisms of $T$-schemes.
\end{itemize}

	\subsubsection*{Acknowledgments}
	
	The authors deeply thank Aravind Asok, Joseph Ayoub, Mikhail Bondarko, Baptiste Calmès, Jean Fasel, Marc Levine, Johannes Nagel, Paul Arne \O stv\ae r, Bertrand Toën for conversations, exchanges and ideas that led to the present paper.

	The work of the first and second-named authors are supported by the ANR HQDIAG project no ANR-21-CE40-0015. The second-named author is supported by the ANR LabEx CIMI within the French State Programme “Investissement d’Avenir”.
	The third-named author is supported by the National Key Research and Development Program of China Grant Nr. 2021-YFA1001400, the National Natural Science Foundation of China Grant Nr. 12101455 and the Fundamental Research Funds for the Central Universities.

\section{Artin approximation}

\begin{paragr}
	
Recall that if $X$ is a scheme, $Z\subset X$ is a subscheme and $x\in Z$, a \textbf{Nisnevich neighborhood of $x$ in $X$} is an \'etale morphism of pointed schemes $(W,w)\to (X,x)$ together with a section $(Z,x)\to (W,w)$. The \emph{Artin approximation property} implies the following result:

\end{paragr}

\begin{thm}[\textrm{\cite[Theorem 1.3]{PopescuNeron}, \cite[Corollary 2.6]{ArtinApprox}}]
	\label{thm:aaprox}
	Let $(S,s)$ be a pointed excellent noetherian scheme. Let $(X,x)$ and $(Y,y)$ be two $(S,s)$-schemes essentially of finite type. If there exists an $\mathcal{O}_S$-linear isomorphism between the completed local rings $\widehat{\mathcal{O}}_{X,x}\simeq\widehat{\mathcal{O}}_{Y,y}$, then the points $x$ and $y$ have a common Nisnevich neighborhood, that is, there exists a diagram of $S$-schemes 
	\begin{align}
		(X,x)\xleftarrow{f}(W,w)\xrightarrow{g}(Y,y)
	\end{align}
	making $(W,w)$ a Nisnevich neighborhood of both $x$ and $y$.
\end{thm}

\begin{paragr}
	
\label{num:deg146}
We now recall a construction in \cite[Theorem 1.4.6]{DegliseOri}. Let $S$ be an excellent noetherian scheme, and let
\begin{align}
	X\xleftarrow{}Z\xrightarrow{}Y
\end{align}
be a diagram of $S$-schemes, where both maps are closed immersions. We assume that 
\begin{itemize}
	\item
	For every point $z\in Z$, there exists an $\mathcal{O}_S$-linear isomorphism between the completed local rings 
	\begin{align}
		\widehat{\mathcal{O}}_{X,z}\simeq\widehat{\mathcal{O}}_{Y,z}.
	\end{align}
	\item
	The closed immersion $Z\to Y$ has a retraction $Y\to Z$, so that we can view $Y$ as a scheme over $X$ via the composition $Y\to Z\to X$.
\end{itemize}

\end{paragr}

\begin{paragr}
	
In the setting of~\ref{num:deg146}, we denote by $\Gamma(X\xleftarrow{}Z\xrightarrow{}Y)$ the category of all commutative diagrams of the form
\begin{align}
	\label{eq:deg146a}
	\begin{split}
		\xymatrix@=10pt{
			& & \mathcal{W} \ar_-{f}[ld] \ar^-{g}[rd] & \\
			\Gamma: & \mathcal{X} \ar_-{p}[d] & \mathcal{Z} \ar^-{}[d] \ar_-{}[u] \ar^-{}[l] \ar^-{}[r] & \mathcal{Y} \ar^-{q}[d] \\
			& X  & Z \ar^-{}[l] \ar^-{}[r] & Y
		}
	\end{split}
\end{align}
where
\begin{itemize}
	\item
	$\mathcal{X}$ is a Zariski hypercovering of $X$, and both squares are Cartesian. 
	\item
	$\mathcal{W}$ is a simplicial $S$-scheme such that on each degree $n\geqslant0$, $\mathcal{W}_n$ is a common Nisnevich neighborhood of $\mathcal{Z}_n$ in $\mathcal{X}_n$ and $\mathcal{Y}_n$. 
\end{itemize}
and morphisms in $\Gamma(X\xleftarrow{}Z\xrightarrow{}Y)$ are morphisms of commutative diagrams which are induced by morphisms of Zariski hypercoverings $\mathcal{X}\to\mathcal{X}'$ and morphisms of Nisnevich neighborhoods $\mathcal{W}\to\mathcal{W}'$.

\end{paragr}

\begin{lm}
	
	\label{lem_non_empty_cofiltered_cat}
	The category $\Gamma(X\xleftarrow{}Z\xrightarrow{}Y)$ is non-empty and cofiltered.
\end{lm}
\proof
For every point $z\in Z$, applying Theorem~\ref{thm:aaprox} to the pointed local $S$-schemes $(\operatorname{Spec}(\mathcal{O}_{X,z}),z)$ and $(\operatorname{Spec}(\mathcal{O}_{Y,z}),z)$, we see that $z$ has a common Nisnevich neighborhood $W_{(z)}$ in these two schemes; such a property can be lifted to some Zariski neighborhood of $z$, that is, there exists a commutative diagram of $S$-schemes
\begin{align}
	\label{eq:deg1461}
	\begin{split}
		\xymatrix@=10pt{
			& W_z \ar_-{f_z}[ld] \ar^-{g_z}[rd] & \\
			X_z \ar_-{}[d] & Z_z \ar^-{}[d] \ar_-{}[u] \ar^-{}[l] \ar^-{}[r] & Y_z \ar_-{}[d] \\
			X  & Z \ar^-{}[l] \ar^-{}[r] & Y
		}
	\end{split}
\end{align}
where $X_z$, $Y_z$, $Z_z$ are Zariski open neighborhoods of $z$ in $X$, $Y$, $Z$, both squares are Cartesian, and the morphisms $f_z$ and $g_z$ are Nisnevich neighborhoods of $Z_z$. 

The diagram~\eqref{eq:deg1461} for all points $z\in Z$ can be assembled into a commutative diagram of the form~\eqref{eq:deg146a} by taking the \v{C}ech nerve, which gives an element in $\Gamma(X\xleftarrow{}Z\xrightarrow{}Y)$. Therefore the category $\Gamma(X\xleftarrow{}Z\xrightarrow{}Y)$ is non-empty.

It remains to show that this category is cofiltered. Arguing Zariski locally, we may assume that $X$ and $Y$ are local with a common closed point $Z=z$. If $X\xleftarrow{}W\xrightarrow{}Y$ and $X\xleftarrow{}W'\xrightarrow{}Y$ are two common Nisnevich neighborhoods of $z$, then there exists an $\mathcal{O}_S$-linear isomorphism between the completed local rings $\widehat{\mathcal{O}}_{W,z}\simeq\widehat{\mathcal{O}}_{W',z}$, and by Theorem~\ref{thm:aaprox} both $W$ and $W'$ can be refined by a third Nisnevich neighborhood $W\xleftarrow{}W''\xrightarrow{}W'$. Same arguments as above show that the category $\Gamma(X\xleftarrow{}Z\xrightarrow{}Y)$ is cofiltered.
\endproof

\section{Purity and specialization}

\begin{paragr}
	\label{Hyp1_purity}
	\textbf{Hypothesis 1}
	
	Let $X$ be a smooth scheme over $S$, $t:X \to \AA^1_S$ a morphism of $S$-schemes such that $Z:=V(t):=t^{-1}({0})$ is smooth over $S$.

\end{paragr}

\begin{paragr}
	\label{Hyp2_purity}
	\textbf{Hypothesis 2}
	
	Let $X$ be a smooth scheme over $S$, $t:X \to \AA^1_S$ a morphism of $S$-schemes such that $Z:=V(t):=t^{-1}({0})$ is smooth over $S$ and equipped with a retraction $r_X:X\to Z$.

\end{paragr}

\begin{lm}
		
	Let $X$ be a smooth scheme over $S$, $t:X \to \AA^1$ a morphism of $S$-schemes such that $Z:=V(t):=t^{-1}({0})$ is smooth over $S$.
	
	There is a retraction $\rho : \widehat{X}_Z \to Z$ of the induced immersion $Z \to \widehat{X}_Z$.
	\par Denote by $\psi_t : N_Z X \to \AA^1_Z$ the isomorphism of schemes. 
	\par We have induced isomorphisms
	\begin{center}
	$	\xymatrix{
	\widehat{X}_Z
	\ar[r]^{\simeq}
	&
	\widehat{\AA^1_Z}
	\ar[r]^{\simeq}_{\psi_t^{-1}}
	&
	\widehat{\NN_ZX}	
	}$

	\end{center}
of $\AA^1_S$-schemes.
\end{lm}

\begin{df}
Let $(X,Z)$ be as in \ref{Hyp1_purity} (resp. $(X,Z,r_X)$ as in \ref{Hyp2_purity}) and put $N:= N_ZX$ (resp. equipped with it canonical projection $r_N:N\to Z$).	We let $\mathscr{I}(X,Z)$ (resp. $\mathscr{I}(X,Z,r_X)$) be the category made of the following data in $\Smsm_{\AA^1_S}$:
	$$
	\xymatrix@=10pt{
		& \cW_\bullet
		\ar_{p_1}[ld]
		\ar^{p_2}[rd] 
		& 
		\\
		\cU_\bullet
		\ar_{q_1}[d] 
		& \cZ_\bullet
		\ar@{^(->}[r]
		\ar@{^(->}[u]
		\ar@{^(->}[d]
		\ar@{_(->}[l] 
		& \cN_\bullet
		\ar^{q_2}[d] 
		\\
		X 
		& Z\ar@{^(->}[r]
		\ar@{_(->}[l] 
		& N
	}
	$$
	where and
	\begin{itemize}
		\item $\cZ_\bullet=Z \times_X \cU_\bullet$, 
		

		\item $q_1$ (and therefore $q_2$) is a Zariski hyper-cover,
		\item for each integer $n \geq 0$, $(\cW_n,\cZ_n) \rightarrow (\cU_n,\cZ_n)$
		and $(\cW_n,\cZ_n) \rightarrow (\cN_n,\cZ_n)$ is an excisive morphism of smooth closed $S$-pairs.
		\item (resp. we have retractions $\cU_\bullet \to \cZ_\bullet$, $\cN_\bullet \to \cZ_\bullet$, $\cU_\bullet \to \cW_\bullet$, $\cN_\bullet \to \cW_{\bullet}$ which are compatible with the previous data).
		
	\end{itemize}
\end{df}

\begin{lm}
Let $(X,Z)$ be as in \ref{Hyp1_purity} (resp. $(X,Z,r_X)$ as in \ref{Hyp2_purity}).
The category $\mathscr{I}(X,Z)$ (resp. $\mathscr{I}(X,Z, r_X)$) is non-empty and cofiltered.
\end{lm}
\begin{proof}
	This is a special case of Lemma \ref{lem_non_empty_cofiltered_cat}.
\end{proof}

\begin{paragr}
	\begin{itemize}		
		\item Denote by $\cC$ the $\infty$-category of pro-objects of simplicial Nisnevich sheaves over $\Smsm_S$.
		\item For $(X,Z)$ as in \ref{Hyp1_purity} (resp. $(X,Z,r_X)$ as in \ref{Hyp2_purity}), we denote by 
\begin{align*}
&\wp(X,Z) = \plim_{\cW_{\bullet} \in \mathscr{I}(X,Z)} \cW_{\bullet}/(\cW_{\bullet} - \cZ_{\bullet}), \\
\text{resp. } & \wp(X,Z,r_X) = \plim_{\cW_{\bullet} \in \mathscr{I}(X,Z,r_X)} \cW_{\bullet}/(\cW_{\bullet} - \cZ_{\bullet}) 
\end{align*}
		considered as an object in $\cC$. 	
\end{itemize}

Fix a diagram $(\cU_\bullet, \cW_\bullet, \cN_\bullet)$ in $\mathscr{I}(X,Z)$ (resp. in $\mathscr{I}(X,Z,r_X)$. We have the following diagram in $\cC$:
\begin{align*}
& \xymatrix{
X/(X-Z)\ar[r]^{(\operatorname{HR})}\ar@{..>}[d]^{\mathfrak{p}_{X,Z}}
 & \cU_\bullet/ (\cU_\bullet - \cZ_{\bullet})\ar@{=}[r]
 & X/(X-Z) \times \cU_\bullet \\
\Th(N_ZX) &	\cN_{\bullet}/ ( \cN_{\bullet} - \cZ_{\bullet})\ar[l]
 & \wp(X,Z)\ar[u]^{p_1}_{\simeq}\ar[l]_{p_2}^{\simeq}
} \\
\text{resp. } & 
\xymatrix{
X/(X-Z)\ar[r]^{(\operatorname{HR})}\ar@{..>}[d]^{\mathfrak{p}_{X,Z,r_X}}
 & \cU_\bullet/ (\cU_\bullet - \cZ_{\bullet})\ar@{=}[r]
 & X/(X-Z) \times \cU_\bullet \\
\Th(N_ZX) &	\cN_{\bullet}/ ( \cN_{\bullet} - \cZ_{\bullet})\ar[l]
 & \wp(X,Z,r_X)\ar[u]^{p_1}_{\simeq}\ar[l]_{p_2}^{\simeq}
}
\end{align*}	
%
%
%

where the hyper-cover map $\operatorname{(HR)}$ is a weak equivalence in the $\infty$-category $\cC$, the maps $p_1$ and $p_2$ are isomorphisms of pro-objects in $\cC$, and ${\mathfrak{p}_{X,Z}}$ (resp. ${\mathfrak{p}_{X,Z, r_X}}$) is the induced map making the diagram commutative. 

\par Denote by $\Delta^{op}\Sh_{Nis}(\Smsm_S)$ the $\infty$-category of simplicial Nisnevich sheaves over $\Smsm_S$. Since the canonical functor from $\Delta^{op}\Sh_{Nis}(\Smsm_S)$ to $\cC$ is fully faithful, we thus obtain the follow theorem which generalizes Morel-Voevodsky purity theorem.

\end{paragr}

\begin{thm}[General purity theorem]	
	\label{thm_general_purity} Let $(X,Z)$ be as in \ref{Hyp1_purity} (resp. $(X,Z,r_X)$), there exists a weak equivalence 
\begin{align*}
\mathfrak{p}_{X,Z}:X/(X-Z) &\simeq \Th(N_ZX) \\
\text{resp. } \mathfrak{p}_{X,Z,r_X}:X/(X-Z) &\simeq \Th(N_ZX)
\end{align*}
	in the $\infty$-category $\Delta^{op}\Sh_{Nis}(\Smsm_S)$.
\end{thm}

\begin{rem}
\begin{enumerate}
	\item More generally, the map $\mathfrak{p}_{X,Z}$ exists for any $i:Z \to X$, regular closed immersion of smooth schemes over $S$.
	\item 	We emphasize the fact no $\AA^1$-homotopy invariance property has been used so far.
	\item In another work, we hope to extend this theorem to non-smooth scheme $X/S$.
\end{enumerate}
\end{rem}

\begin{df}
	
	\label{def_specialization}
	Let $(X,Z)$ as in \ref{Hyp1_purity} (resp. $(X,Z,r_X)$ as in \ref{Hyp2_purity}). We define a specialization map $\spz_{X,Z}$ (resp. $\spz_{X,Z,r_X}$) as the unique map making the following diagram commutative
\begin{align*}
&\xymatrix{
	\Th(\AA^1_Z)
	\ar[r]^{\psi_t^{-1}}
	\ar@{..>}[d]^{\spz_{X,Z}}
		&
		\Th (N_ZX)
		&
		X/(X-Z)
		\ar[l]_{\simeq}^{\mathfrak{p}_{X,Z}}
		\ar[d]
		\\
		\Th(\AA^1_Z)\wedge (X-Z)_{+}
		&
		S^1 \wedge (\Gm \wedge (X-Z) )_{+}
		\ar[l]^{\simeq}
		&
		S^1 \wedge (X-Z)_{+}
		\ar[l]_-{\gamma_t}	
	} \\
\text{resp. } &
\xymatrix{
		\Th(\AA^1_Z)
		\ar[r]^{\psi_t^{-1}}
		\ar@{..>}[d]^{\spz_{X,Z,r_X}}
		&
		\Th (N_ZX)
		&
		X/(X-Z)
		\ar[l]_{\simeq}^{\mathfrak{p}_{X,Z,r_X}}
		\ar[d]
		\\
		\Th(\AA^1_Z)\wedge (X-Z)_{+}
		&
		S^1 \wedge (\Gm \wedge (X-Z) )_{+}
		\ar[l]^{\simeq}
		&
		S^1 \wedge (X-Z)_{+}
		\ar[l]_-{\gamma_t}	
	}
\end{align*}
\end{df}

\begin{paragr}
	Denote by $\cH$ the set of pullbacks of hyper-coverings and by $\Span_{\cH}(\cC)$ the category whose objects are the objects of $\cC$ and whose maps are spans
	\begin{center}
		$
		\xymatrix{
	{}
	&
	R
	\ar[dr]^q
	\ar[dl]_p
	&
	{}
	\\
	P
	&
	{}
	&
	Q	
	}$

	\end{center}

where $p \in \cH$.
\end{paragr}

\begin{paragr}
	\label{paragr_house_of_spans}
	We have to following commutative diagram of $\infty$-categories
	
	\begin{center}
		$
		\xymatrix{
			{}
			&
			\Span_{H}(\cC)
			\ar[rd]^{L_{\textbf{HR}, \AA^1 }}
			\ar[ld]_{L_{\textbf{HR}}}
			&
			{}
			\\
			L_{\textbf{HR}}\Span_{H}(\cC)
			\ar[rr]^{L_{\AA^1 }}
			&
			{}
			&
			L_{\textbf{HR}, \AA^1 }\Span_{H}(\cC)
			\\
			L_{\textbf{HR}}\cC 
			\ar[u]^{\simeq}
			\ar[rr]
			&
			{}
			&
			L_{\textbf{HR}, \AA^1 }\cC
			\ar[u]^{\simeq}
		}$
		
	\end{center}
	
	where $L_{\textbf{HR}}$ (resp. ${L_{\AA^1 }}$) denotes the localization with respect to the Zariski hyper-coverings (resp. the projection maps $\AA^1_X \to X$) and the vertical maps are the canonical inclusions.

\end{paragr}

\begin{rem}
	The maps $\mathfrak{p}_{X,Z}$ and $\spz_{X,Z}$ are in $\Span_{\cH}(\cC)$.
\end{rem}

\begin{df}
	\label{def_cat_smsp}
	We define, and denote by $\Smsp_S$ (resp. $\Smspr$), the free category whose objects are the smooth $S$-schemes, and whose set of morphisms is generated by:
	\begin{itemize}
		\item [\namedlabel{itm:sp1}{(Dsp1)}] a symbol $[f]:Y \to X$ for each smooth morphisms of smooth $S$-schemes $f:Y\to X$,
		\item [\namedlabel{itm:sp2}{(Dsp2)}] a symbol $[\spz_{X,Z}]:Z \to X-Z$ (resp. $[\spz_{X,Z,r_X}]:Z \to X-Z$) for couple $(X,Z)$ as in \ref{Hyp1_purity} (resp. triple $(X,Z,r_X)$ as in \ref{Hyp2_purity}),
			\end{itemize}
	modulo the relations:
	\begin{itemize}
		\item [\namedlabel{itm:Rsp1}{(Rsp1)}] For any $f,g$ composable smooth morphisms of smooth $S$-schemes, we have 
		\begin{center}
			$[f \circ g]=[f] \circ [g]$,
		\end{center}
		and, for any smooth $S$-scheme $X$, we have $[\Id_X]=\Id_X$.
		\item [\namedlabel{itm:Rsp2}{(Rsp2)}] Let $(X,Z)$ as in \ref{Hyp1_purity} (resp. $(X,Z,r_X)$ as in \ref{Hyp2_purity}), and let $f:Y\to X$ be a smooth morphism of smooth $S$-schemes. Denote by $t'=t\circ f$ (hence $T:=V(t')$ is smooth over $S$), and set $g:T \to Z$ and $h:Y-T \to X-Z$ the induced maps. (resp. Denote by $r_Y$ the pullback of $r_X$ along $f$). We have the relation
		\begin{align*}
			[\spz_{X,Z}] \circ [g] &= [h] \circ [\spz_{Y,T}] \\
			\text{resp. } [\spz_{X,Z,r_X}] \circ [g] &= [h] \circ [\spz_{Y,T, r_Y}].
		\end{align*}
	\item [\namedlabel{itm:Rsp3}{(Rsp3)}] For $(X,Z,r_X)$ as in \ref{Hyp2_purity}. Let $r_{X-Z}:(X-Z) \rightarrow Z$ be the restriction of $r_X$.	We also think of $Z$ as a closed subscheme of $\AA^1_Z$ and consider the canonical projection $r_{\AA^1_Z}:\AA^1_Z \to Z$ and its restriction $r_{{\Gm}_Z}:{\Gm}_Z \to Z$. 
			We add the relation:
			\begin{align*}
				[r_{X-Z}] \circ [\spz_{X,Z}] & = [r_{\AA^1_Z}] \circ [\spz_{\AA^1_Z,Z}] \\
				\text{resp. } [r_{X-Z}] \circ [\spz_{X,Z,r_X}] & = [r_{\AA^1_Z}] \circ [\spz_{\AA^1_Z,Z,r_{{\Gm}_Z}}]
			\end{align*}
\end{itemize}	
\end{df}

\begin{rem}\label{rem:specialisations}
	In Relation \ref{itm:sp2}, the divisor $Z$ can be empty.
	Moreover, we assume that the data of the pair $(X,Z)$ (resp. $(X,Z,r)$) is part of the given morphism $[\spz_{X,Z}]$ (resp. $[\spz_{X,Z,r}]$).
\end{rem}

\begin{rem}
	An immediate consequence of Definition \ref{def_cat_smsp} is the existence of a canonical functor
	\begin{center}
		${\iota^{sp}_{spr}} : \Smspr_S \to \Smsp_S$	\end{center}
	that forgets the retraction.
\end{rem}

\begin{paragr}
	
	\label{paragr_smsp_to_spans}
	We define a canonical functor 
	\begin{align*}
		& \Theta_{sp}: \Smsp_S \to L_{\textbf{HR},\AA^1}\Span_{\cH}(\cC) \\
		\text{resp. } & \Theta_{spr}: \Smspr_S \to L_{\textbf{HR},\AA^1}\Span_{\cH}(\cC)
	\end{align*}
	as follows:
	\begin{enumerate}
	\item[\ref{itm:sp1}] A symbol $[f]:Y \to X$ (coming from a smooth morphisms of smooth $S$-schemes $f:Y\to X$) is send to (the trivial span defined by) $f$,
		\item[\ref{itm:sp2}] A symbol $[\spz_{X,Z}]:Z \to X-Z$ (for $(X,Z)$ as in \ref{Hyp1_purity}) is send to the specialization map $\spz_{X,Z}$ defined in \ref{def_specialization}.
	\\ 	(resp.
		A symbol $[\spz_{X,Z,r_X}]:Z \to X-Z$ (for $(X,Z,r_X)$ as in \ref{Hyp2_purity}) is send to the specialization map $\spz_{X,Z,r_X}$ defined in \ref{def_specialization}.
		)

		\item[\ref{itm:Rsp1}] This relation is satisfied trivially.

		\item[\ref{itm:Rsp2}]
		Let $(X,Z)$ as in \ref{Hyp1_purity}, and let $f:Y\to X$ be a smooth morphism of smooth $S$-schemes. Denote by $t'=t\circ f$ (hence $T:=V(t')$ is smooth over $S$), and set $g:T\to  Z$ and $h:Y-T \to X-Z$ the induced maps. In order to prove that
		\begin{center}
			$\spz_{X,Z} \circ g = h \circ \spz_{Y,T}$		
		\end{center}
in $L_{\textbf{HR},\AA^1}\Span_{\cH}(\cC)$, it suffices to work in $\Span_{\cH}(\cC)$.
	Consider an object in $\mathscr{I}(X,Z)$ given by the diagram
		$$
	\xymatrix@=10pt{
		& \cW_\bullet
		\ar_{p_1}[ld]
		\ar^{p_2}[rd] 
		& 
		\\
		\cU_\bullet
		\ar_{q_1}[d] 
		& \cZ_\bullet
		\ar@{^(->}[r]
		\ar@{^(->}[u]
		\ar@{^(->}[d]
		\ar@{_(->}[l] 
		& \cN_\bullet
		\ar^{q_2}[d] 
		\\
		X 
		& Z\ar@{^(->}[r]
		\ar@{_(->}[l] 
		& N
	}
	$$
	and consider its pullback along $Y\to X$ which is an object in $\mathscr{I}(Y,T)$ denoted by 
		$$
	\xymatrix@=10pt{
		& \cW'_\bullet
		\ar_{p_1'}[ld]
		\ar^{p_2'}[rd] 
		& 
		\\
		\cU'_\bullet
		\ar_{q_1'}[d] 
		& \cZ'_\bullet
		\ar@{^(->}[r]
		\ar@{^(->}[u]
		\ar@{^(->}[d]
		\ar@{_(->}[l] 
		& \cN'_\bullet
		\ar^{q_2'}[d] 
		\\
		Y 
		& T\ar@{^(->}[r]
		\ar@{_(->}[l] 
		& N'
	}
	$$
	where, by assumptions, we have $T\simeq Y \times_X Z$ and 
	$$
	N'\simeq  N_TY \simeq N_Z X \times_Z T \simeq N_Z X \times_X Y.
	$$
	
The following diagram
\begin{center}
	$\xymatrix{
	(Y,T)
	\ar[r]
	\ar[d]
	&
	(X,Z)
	\ar[d]
	\\
	(\cU'_\bullet, \cZ'_\bullet)
	\ar[r]
	\ar[d]
	&
	(\cU_\bullet, \cZ_\bullet)
	\ar[d]
	\\
	(\cW'_\bullet, \cZ'_\bullet)
	\ar[r]
	\ar[d]
	&
	(\cW_\bullet,\cZ_\bullet)
	\ar[d]
	\\
	(\cN'_\bullet, \cZ'_\bullet)
	\ar[r]
	\ar[d]
	&
	(\cN_\bullet,  \cZ'_\bullet)
	\ar[d]
	\\
	(N',T)
	\ar[r]
	&
	(N,Z)
	}$

\end{center}
is commutative.

Thus, the following diagram
\begin{center}
	$\xymatrix{
Y/(Y-T)
\ar[r]
&
X/(X-Z)
\\
\cW'_\bullet/(\cW'_\bullet - \cZ'_\bullet)
\ar[u]
\ar[r]
\ar[d]
&
\cW_\bullet/(\cW_\bullet - \cZ_\bullet)
\ar[u]
\ar[d]
\\
\Th(N')
\ar[r]
&
\Th(N)
	}
$

\end{center}

is also commutative, and we can define the following composite map 
\begin{center}
$
\xymatrix{
\plim_{\cW'_{\bullet} \in \mathscr{I}(Y,T)}
\cW'_\bullet/(\cW'_\bullet - \cZ'_\bullet)
\ar[d]
\ar@/^10pc/[dd]^{\alpha_f}
\\
\plim_{\cW'_{\bullet} \in \mathscr{I}(Y,T)}
\Phi(\cW_\bullet)/(\Phi(\cW_\bullet) - \Phi(\cZ_\bullet))
\ar[d]
\\
\plim_{\cW_{\bullet} \in \mathscr{I}(X,Z)}
	\cW_\bullet/(\cW_\bullet - \cZ_\bullet)
}
$
\end{center}
where we have denoted by
$$
\Phi: \mathscr{I}(X,Z)
\to
\mathscr{I}(Y,T)
$$
the canonical functor induced by pullbacks along $Y\to X$.

According to Definition \ref{def_specialization}, the following diagram
\begin{center}
	$
	\xymatrix{
	\Th(\AA^1_T)
	\ar[r]
	\ar[d]
	\ar@/_9pc/[ddddd]^{\spz_{Y,T}}
	&
	\Th(\AA^1_Z)
	\ar@/^9pc/[ddddd]_{\spz_{X,Z}}
	\ar[d]
	\\
	\Th(N_TY)
	\ar[r]
	&
	\Th(N_ZX)
	\\
	Y/(Y-Z')
	\ar[r]
	\ar[d]
	\ar[u]
	&
	X/(X-Z)
	\ar[d]
	\ar[u]
	\\
	S^1 \wedge (Y-T)_{+}
	\ar[d]^{\gamma_{t'}}
	\ar[r]
	&
	S^1 \wedge (X-Z)_{+}
	\ar[d]^{\gamma_{t}}
	\\
		S^1 \wedge (\Gm \wedge (Y-T) )_{+}
		\ar[r]
		\ar[d]^{\simeq}
	&
	S^1 \wedge (\Gm \wedge (X-Z) )_{+}
	\ar[d]^{\simeq}
	\\
	\Th(\AA^1_T)
	\wedge (Y-T)_{+}
	\ar[r]
	&
		\Th(\AA^1_Z)
	\wedge (X-Z)_{+}
}$
\end{center}
is commutative where the horizontal maps are induced by $f$ (and $\alpha_f$).

Putting things together, we conclude that:
	\begin{center}
	$\spz_{X,Z} \circ g = h \circ \spz_{Y,T}$.			
\end{center}	
Respectively for $\Smspr_S$, we can prove verbatim that 
	\begin{center}
	$\spz_{X,Z,r_X} \circ g = h \circ \spz_{Y,T,r_Y}$			
\end{center}
where we use the notations of \ref{itm:Rsp2}.
		
\item[\ref{itm:Rsp3}] For this relation, we need the $\AA^1$-localization property.
We start by working on the case $\Smspr_S$, that is we prove:
		\begin{center}
	$[r_{X-Z}] \circ [\spz_{X,Z,r_X}] = [r_{\AA^1_Z}] \circ [\spz_{\AA^1_Z,Z,r_{{\Gm}_Z}}]$
\end{center}
where we use the notations of \ref{itm:Rsp3}.
This relation follows from the fact that the diagram

\begin{center}	
	\resizebox{!}{10pt}{
	$
	\xymatrix{
		X/(X-Z)
		\ar[d]
		&
		\wp(X,Z,r_X)
		\ar[l]
		\ar[d]
		\ar[r]
		&
		\Th(N_ZX)
		\ar[d]
		\ar[r]
		&
		\Th(\AA^1_Z)
		\ar[d]
		\\
		S^1
		\wedge 
		(X-Z)_{+}
		\ar[d]^-{\gamma_{t}}
		&
	\plim_{\cW_{\bullet} \in \mathscr{I}(X,Z,r_X)}	S^1
		\wedge
		(X-\cW_{\bullet})
		\ar[r]
		\ar[d]^-{\gamma_{t}}
		\ar[l]
		&
		S^1
		\wedge
		(N_ZX - Z)_{+}
		\ar[r]
		\ar[d]^-{\gamma_{t}}
		&
		S^1
		\wedge
		{\Gm}_{Z,+}
		\ar[d]^-{\gamma_{t}}
		\\
		S^1
		\wedge
		\Gm 
		\wedge
		(X-Z)_{+}
		\ar[d]^{r_{X-Z}}
				&
	\plim_{\cW_{\bullet} \in \mathscr{I}(X,Z,r_X)}	S^1
	\wedge
	\Gm 
	\wedge
	(\cW_{\bullet}-Z)_{+}
	\ar[d]
	\ar[l]
	\ar[r]
	&
		S^1
	\wedge
	\Gm 
	\wedge
	(N_ZX-Z)_{+}
	\ar[d]^{r_{N_ZX-Z}}
	\ar[r]	
	&
		S^1
	\wedge
	\Gm 
	\wedge
	{\Gm}_{Z,+}
	\ar[d]
	\\
		S^1
	\wedge
	\Gm 
	\wedge
	Z_{+}
	\ar@{=}[r]
	&
		S^1
	\wedge
	\Gm 
	\wedge
	Z_{+}
	\ar@{=}[r]
	&
		S^1
	\wedge
	\Gm 
	\wedge
	Z_{+}
	\ar@{=}[r]
	&
		S^1
	\wedge
	\Gm 
	\wedge
	Z_{+}
	}
$
}
\end{center}

	is commutative. 
	\par We now work on the case $\Smsp_S$. In order to prove that 
		\begin{center}
		$[r_{X-Z}] \circ [\spz_{X,Z}] = [r_{\AA^1_Z}] \circ [\spz_{\AA^1_Z,Z}]$,
	\end{center}
we can draw a similar diagram as above but one has to be careful of the fact that the middle-bottom square may not commute in this case. The solution is to reduce to the previous case by considering the commutative diagram:
\begin{center}
	$
	\xymatrix{
			X/(X-Z)
		\ar[d]^{\simeq}
		&
		\wp(X,Z,r_X)
		\ar[l]
		\ar[d]^{\simeq}
		\ar[r]
		&
		\Th(N_ZX)
		\ar[d]^{\simeq}
		\ar[r]
		&
		\Th(\AA^1_Z)
		\ar[d]^{\simeq}
		\\
			X/(X-Z)
		&
		\wp(X,Z)
		\ar[l]
		\ar[r]
		&
		\Th(N_ZX)
		\ar[r]
		&
		\Th(\AA^1_Z)
					}
$

\end{center}
where the vertical maps are induced by the fact the canonical functor $\mathscr{I}(X,Z,r_X) \to \mathscr{I}(X,Z)$ is fully faithful, and thus are isomorphisms thanks to the $\AA^1$-localization property.
\end{enumerate}
\end{paragr}

\section{Moving lemma via the smooth-smooth site}

\begin{paragr}
	\label{paragr_magic_functor}
	Putting \ref{paragr_house_of_spans} and \ref{paragr_smsp_to_spans} together, we obtain an $\infty$-functor
	\begin{center}
		$\mathscr{M}: \Smsp_S \to \SH(\Smsm_S)
		$
	\end{center}
	where  $\SH(\Smsm_S)$ is the stable motivic homotopy $\infty$-category defined over the smooth-smooth site $\Smsm_S$.
\end{paragr}

\begin{paragr}
	We may now prove our main Theorem \ref{thm_main_moving}. We are going to construct a quasi-inverse to the functor
	$$
	\nu^*:\SH(\Sm_S) \rightarrow \SH(\Smsm_S)
	$$
	by choosing models. Let $\E \in \SH(\Sm_S)$ be a motivic spectrum which can be seen as a family of functors 
	$$
	\E_n: (\Smsm_S)^{op} \to \mathscr{S}
	$$
	where $\mathscr{S}$ is the $\infty$-category of spaces.
	\par We can extend $\E_n$ to a functor
	$$
	\widehat{\E}_n: \SH(\Smsm_S) \to \mathscr{S}
	$$
	
	by taking colimits and compose this with 
	\begin{center}
		$\mathscr{M}: \Smsp_S \to \SH(\Smsm_S)
		$
	\end{center}
	in order to obtain an $\infty$-functor
	\begin{center}
		$\widehat{\E}_n: \Smsp_S \to \mathscr{S}
		$
	\end{center}
	for each natural number $n$.
	\par Take a model still denoted by $\widehat{\E}_*$ and rigidify it by setting
	$$
	\tilde{\E}_*(X)
	=
	\lim_{r\in \NN}
	\widehat{\E}(\AA^r_S\times_S X)
	$$
	for any smooth scheme $X/S$. Exactly as in \cite[Section 2.2]{DegliseFeldJin22}, we can prove that $\tilde{\E}_*$ can be equipped with Gysin morphisms for local complete intersection maps. Since we can work with affine schemes without loss of generality, this implies that $\tilde{\E}_*$ can be extended to the whole site $\Sm_S$. Denote by $\Psi(\E_*) \in \SH(\Sm_S)$ this new object. By construction, the restriction of $\Psi(\E_*)$ to the smooth-smooth site $\Smsm_S$ is exactly $\E_*$.
	
	\par Conversely, if $\E'$ is an element in $\SH(\Sm_S)$, then we can consider its restriction $\nu^*(\E')$ to the smooth-smooth site. In this case, we can see that, for $(X,Z)$ as in \ref{Hyp1_purity}, the purity morphism
	$$
	\mathfrak{p}_{X,Z}^*:
	\nu^*(\E')(\Th(N_ZX))
	\to 
	\nu^*(\E')(X/(X-Z))
	$$
	defined in \ref{thm_general_purity} coincides with Morel-Voevodsky purity morphism. Hence, the pullbacks constructed on $\Psi(\nu^*(\E'))$ correspond to the pullbacks on $\E'$, thus
	$$
	\Psi(\nu^*(\E'))
	\simeq
	\E'
	$$
	which concludes the proof of Theorem \ref{thm_main_moving}.
\end{paragr}

%
%
%

\bibliographystyle{amsalpha}
\bibliography{MW}

\end{document}